\documentclass{amsart}

\usepackage[latin1]{inputenc}
\usepackage{amsfonts,amssymb,amsmath}
\usepackage{verbatim}
\input{xypic}

\newtheorem{theorem}{Theorem}
\newtheorem{proposition}[theorem]{Proposition}
\newtheorem{lemma}[theorem]{Lemma}

\newtheorem{example}[theorem]{Example}

\theoremstyle{remark}
\newtheorem{remark}[theorem]{Remark}

\renewcommand\leq{\leqslant}
\newcommand\Z{\ensuremath{\mathbb Z}}
\newcommand\Q{\ensuremath{\mathbb Q}}
\newcommand\C{\ensuremath{\mathbb C}}
\newcommand\Qb{{\overline\Q}}
\newcommand\f{{\mathfrak f}}

\newcommand{\ra}{{\rightarrow}}
\newcommand{\lra}{\longrightarrow}

\newcommand{\M}{{N_{L/\Q}(\cN_L(E))}}

\newcommand\End{\operatorname{End}}

\newcommand\Gal{\operatorname{Gal}}
\newcommand\GL{\operatorname{GL}}
\newcommand\Hom{\operatorname{Hom}}

\newcommand\ord{\operatorname{ord}}

\newcommand\Res{\operatorname{Res}}

\newcommand\acc[2]{\ensuremath{{}^{#1}\hskip-0.1ex{#2}}}

\newcommand{\comp}{\begin{picture}(6,5)(-3,-2)\put(0,1){\circle{2}} \end{picture}}\def\circ{\comp}
\def\cN{\mathcal N}

\title[Modularity level of modular abelian varieties over number fields]{On the modularity level of modular abelian varieties over number fields}
\author{Enrique Gonz\'alez--Jim\'enez}
\address{Universidad Aut{\'o}noma de Madrid, Departamento de Matem{\'a}ticas and Instituto de Ciencias Matem{\'a}ticas (CSIC-UAM-UC3M-UCM), Madrid, Spain}
\thanks{The first author was supported in part by grants MTM 2009-07291 (Ministerio de Educaci{\'o}n y Ciencia, Spain) and  CCG08-UAM/ESP-3906 (Universidad Auton{\'o}ma de Madrid-Comunidad de Madrid, Spain). The second author was supported in part by grants 2009 SGR 1220 and MTM2009-13060-C02-01.}
\email{enrique.gonzalez.jimenez@uam.es}

\author{Xavier Guitart}
\address{Departament de Matem\`atica aplicada II, Universitat Polit\`ecnica de Catalunya, Jordi Girona 1-3 (Edifici Omega) 08034, Barcelona}
\email{xevi.guitart@gmail.com}

\date{\today}
\begin{document}

\begin{abstract}
Let $f$ be a weight two newform for $\Gamma_1(N)$ without complex
multiplication. In this article we study the conductor of the
absolutely simple factors $B$ of the  variety $A_f$ over certain
number fields $L$. The strategy we follow is to compute the
restriction of scalars $\Res_{L/\Q}(B)$, and then to apply Milne's
formula for the conductor of the restriction of scalars. In this
way we obtain an expression for  the local exponents of the
conductor $\cN_L(B)$. Under some hypothesis it is possible to give
global formulas relating this conductor with $N$. For instance, if
$N$ is squarefree we find that $\cN_L(B)$ belongs to $\Z$ and
$\cN_L(B)\,\f_L^{\, \dim B}=N^{\dim B}$, where $\f_L$ is the
conductor of $L$.
\end{abstract}

\maketitle

\section{Introduction}
Let $C$ be an elliptic curve defined over $\Q$. The
Shimura-Taniyama-Weil conjecture, also known as the modularity
theorem after its proof by Wiles et al. \cite{wiles,BCDT} asserts
that there exists a surjective morphism $J_0(N)\ra C$ defined over
$\Q$, where $J_0(N)$ is the Jacobian of the modular curve
$X_0(N)$. Moreover, the minimum $N$ with this property is equal to
$\cN_\Q(C)$, the conductor of $C$.

A generalization of the modularity theorem, which as Ribet showed
in \cite{ribet-avqmf} is a consequence of the recently proved
Serre's conjecture on residual Galois representations,
characterizes the modular abelian varieties over $\Q$; that is,
the $\Q$-simple abelian varieties $A$ defined over $\Q$ with a
surjective morphism $J_1(N)\ra A$. They are the so-called (simple)
varieties of $\GL_2$-type: those whose endomorphism algebra
$\Q\otimes\End_\Q(A)$ is a number field of degree over
$\Q$ equal to $\dim A$. 

From the modular form side, one can start with a weight two
newform $f$ for $\Gamma_1(N)$. A construction of Shimura attaches
to such an $f$ an abelian variety $A_f$ over $\Q$,  which is  a
quotient of $J_1(N)$; in fact, all quotients of $J_1(N)$ over $\Q$
are of this form. For these varieties Carayol \cite{carayol}
proved that $\cN_\Q(A_f)=N^{\dim A_f}$. The generalization of
Shimura-Taniyama-Weil asserts that each abelian variety of
$\GL_2$-type is isogenous over $\Q$ to an $A_f$ for some $f$,
therefore the formula $\cN_\Q(A)=N^{\dim A}$ is valid for all $A$
of $\GL_2$-type.


The modular abelian varieties $A_f$ are simple over $\Q$, but they
are not absolutely simple in general: they are isogenous over
$\Qb$ to a power of an absolutely simple variety, which is called
a building block for $A_f$. To be more precise, if $L$ is the
smallest number field where all the endomorphisms of $A_f$ are
defined, then $A_f$ is isogenous over $L$ to a variety of the form
$B^n$, for some absolutely simple  variety $B$ defined over $L$.
In this article we discuss possible generalizations of Carayol's
formula for these modular abelian varieties over number fields
$B/L$, in the case where they do not have CM.

More concretely, in section \ref{sec: 2} we recall the notation
and  basic facts regarding modular abelian varieties and building
blocks. Afterwards,  we give the explicit decomposition of the
restriction of scalars $\Res_{L/\Q}(B)$ as as product of modular
abelian varieties up to isogeny over $\Q$. We use this in section
\ref{sec: 3} in order to give an expression for the local
exponents of $\cN_L(B)$, in terms of the levels of certain twists
of $f$ by Dirichlet characters related to the field $L$. In some
cases, the conductor $\cN_L(B)$ turns out to be a rational integer
and we obtain   similar formulas to the ones for the varieties
$A_f$; we remark that in this situation  the conductor of $L$,
that we denote by $\f_L$, also appears in the expressions. We have
collected all these formulas, that appear in the text as
propositions \ref{prop: formula conductor nivel ord leq 2},
\ref{prop: formula conductor nivel Gamma0} and \ref{prop:
squarefree}, in the following

\medskip

\noindent {\bf Main Theorem.} {\it  Let $f\in S_2(N,\varepsilon)$
be a weight two newform for $\Gamma_1(N)$ with Nebentypus
$\varepsilon$ and without complex multiplication. Let $A_f$ be the
modular abelian variety attached to $f$,  let $L$ be the smallest
field of definition of the endomorphisms of $A_f$, and let $B/L$
be a simple quotient of $A_f$ over $L$.
\begin{enumerate}
\item Suppose that one of the following conditions is satisfied:
\begin{itemize}
\item $N$ is odd and $\ord(\varepsilon)\leq 2$,
\item $N$ is squarefree.
\end{itemize}
Then $\cN_L(B)$ belongs to $\Z$ and $$\cN_L(B)\,\f_L^{\,\dim
B}=N^{\,\dim B}.$$
\item If $f$ is a
newform for $\Gamma_0(N)$, that is if $\varepsilon=1$, then
$\cN_L(B)$ belongs to $\Z$. Moreover,
\begin{enumerate}
\item  if $v_2(\f_L)=3$ and $v_2(\f_K)=2$ for some $K\subseteq L$ then $$2\ \cN_L(B)\ \f_L^{\,\dim B}=N^{\dim B},$$
\item in the remaining cases for $L$ then $$\cN_L(B)\ \f_L^{\,\dim B}=N^{\dim B}.$$

\end{enumerate}

\end{enumerate}
}
\medskip

Finally, in section \ref{sec: 4}  we provide some examples of
building blocks of dimension one and two with their corresponding
equations. Concretely, for the case of dimension one we compute
their conductors, in order to show the different behaviors when
the hypothesis of the above theorem are not satisfied.
 We observe that, although the conductor can be a rational
 integer sometimes,  formulas as the ones in the
 theorem do not always hold; we also give examples where the conductor is not a rational integer.
For the case of dimension two, the level-conductors local formula
provided at Proposition \ref{cond_local} allows us to compute the
conductor of the Jacobian of a genus two curve defined over a
number field that corresponds to a building block.

\section{Modular abelian varieties}\label{sec: 2}
We begin this section by recalling  the basic facts about modular
abelian varieties and  their absolutely simple factors that we
will use. In particular, we introduce the type of varieties that
we will be dealing with in the rest of the article: the building
blocks. Our goal  is to prove a formula for the restriction of
scalars of building blocks, which will be the base for our
analysis of their conductors in the subsequent sections.

Let $f=\sum a_nq^n\in S_2(N,\varepsilon)$ be a normalized newform
without complex multiplication of weight $2$, level $N$ and
Nebentypus $\varepsilon$, and let $E=\Q\left(\{a_n\}\right)$ and
$F=\Q\left(\{a_p^2\varepsilon(p)^{-1}\}_{p\nmid N}\right)$. These
number fields will be denoted by $E_f$ and $F_f$ if we need to
make  the newform from which  they come from explicit. The
extension $E/F$ is abelian, and for each $s\in \Gal(E/F)$ there
exists a single Dirichlet character $\chi_s$ such that $\acc s f =
f\otimes\chi_s$, where $f\otimes\chi_s$ is a newform whose $p$-th
Fourier coefficient coincides with $a_p\chi_s(p)$ for almost all
$p$ (see \cite[\S 3 ]{ribet-twists}). Since $\acc s f$ has level
$N$, the conductor of $\chi_s$ is divisible only by primes dividing
$N$. We will also consider another number field attached to $f$,
namely $L=\Qb^{\cap\ker \chi_s}$ where $s$ runs through
$\Gal(E/F)$.

Shimura  \cite[Theorem~1]{shimura73} attached to $f$ an abelian
variety $A_f/\Q$  constructed as a quotient of $J_1(N)$, the
Jacobian of the modular curve $X_1(N)$, and with an action of $E$
as endomorphisms defined over $\Q$. In fact,
$\Q\otimes_\Z\End_\Q(A_f)\simeq E$ and since $\dim A_f=[E:\Q]$,
the modular abelian varieties  $A_f$ are of $\GL_2$-type; as a
consequence of Serre's conjecture all varieties of $\GL_2$-type
are isogenous to some $A_f$.

The variety $A_f$ is simple over $\Q$, but it is not necessarily
absolutely simple. In general, $A_f$ is isogenous over $\Qb$ to a
power of an absolutely simple abelian variety $B$, which is called
a \emph{building block} of $A_f$. This $B$ has some remarkable
properties; for instance, it is isogenous to all of its Galois
conjugates. In addition, its endomorphism algebra
$\Q\otimes\End(B)$ is a central division algebra over a number
field isomorphic to $F$, it has Schur index $t=1$ or $t=2$ and its
reduced degree $t[F:\Q]$ is equal to $\dim B$. The building blocks
of dimension one are the $\Q$-curves, i.e. the  elliptic curves
$B/\Qb$ that are isogenous to all of their Galois conjugates.

There are infinitely many  varieties $A_f$ with the same
absolutely simple factor up to isogeny. However, by a result of
Ribet \cite[Theorem 4.7]{ribet-twists} if this happens for two
varieties $A_f$ and $A_g$ we can suppose that $g$ is the twist of
$f$ by some Dirichlet character. We will need a more precise
version of this result, which already appears implicitly in
Ribet's proof.

\begin{proposition}\label{corolario ribet}
Let $f,g$ be two normalized newforms without complex
multiplication such that $A_f\sim_K B^n$ and $A_g\sim_K B^m$ for
some absolutely simple abelian variety $B$ over a number field
$K$. Then there exists a character $\chi:\Gal(K/\Q)\ra \C^\times$
such that $g=f\otimes \chi$.
\end{proposition}
\begin{proof}
Let $V_f=T_\ell(A_f)\otimes\Q$ and $V_g=T_\ell(A_{g})\otimes\Q$,
where $T_\ell(A_f)$ and $T_\ell(A_g)$ are the Tate modules
attached to $A_f$ and $A_g$ respectively, and let $H=\Gal(\Qb/K)$.
Under the hypothesis of the proposition we have
$\Hom_H(V_f,V_g)\neq 0$. By \cite[Theorem 4.7]{ribet-twists},
there exists a character $\chi:G_\Q\ra\C^\times$ such that
$f=g\otimes \chi$. Ribet already asserts in the proof of his
theorem that $\chi$ is necessarily trivial on $H$. Therefore,
$\chi$ comes from a character $\chi:\Gal(K/\Q)\ra \C^\times$.
\end{proof}

Gonz\'alez and Lario proved in \cite{pep-lario} that $L$ is the
smallest number field where all the endomorphisms of $A_f$ are
defined. This implies that $A_f\sim_L B^n$, with $B/L$ a building
block with the endomorphisms defined over $L$ and which is
$L$-isogenous to all of its Galois conjugates. From now on $B$
will denote such a building block obtained by decomposing $A_f$
over the field $L$ defined above, and $t$ will denote the Schur
index of $\End(B)$. Using the results in \cite{guitart-quer} one
can show that the restriction of scalars $\Res_{L/\Q}(B)$ is
isogenous over $\Q$ to a product of modular abelian varieties.
Indeed, one has the following

\begin{proposition}\label{prop: restriccion de escalares de B primera proposicion}
The restriction of scalars $\Res_{L/\Q}(B)$ decomposes into simple
abelian varieties up to isogeny as \begin{equation}\label{eq:
restricion de escalares B como producto de Afs}
\Res_{L/\Q}(B)\sim_\Q \prod \left( A_{f_1}\right)^t\times
\cdots\times \left( A_{f_r}\right)^t,
\end{equation}
where $A_{f_1},\dots,A_{f_r}$ are non-isogenous modular abelian
varieties.
\end{proposition}
\begin{proof}
There exists a map $\alpha: \Gal(L/\Q)\ra E^\times$ such that
$\acc
\sigma\varphi=\alpha(\sigma)\circ\varphi\circ\alpha(\sigma)^{-1}$.
The $\alpha(\sigma)$ can be identified with an element of
$\Q\otimes\End_\Q(A_f)$, and together with an isogeny $A_f\sim_L
B^n$ it can be used to construct an isogeny $\mu_\sigma\colon \acc
\sigma B\ra B$ compatible with $\End(A_f)$ (see \cite[Proposition
1.5]{pyle}). It turns out that $\mu_\sigma\circ\acc \sigma
\mu_\tau\circ
\mu_{\sigma\tau}^{-1}=\alpha(\sigma)\circ\alpha(\tau)\circ\alpha(\sigma\tau)^{-1}$
as elements of the center of $\Q\otimes\End(B)$. Therefore, the
cocycle $c_{B/L}(\sigma,\tau)=\mu_\sigma\circ\acc \sigma
\mu_\tau\circ \mu_{\sigma\tau}^{-1}$ is symmetric, since
$\Gal(L/\Q)$ is abelian. Now \cite[Theorem 5.3]{guitart-quer}
implies that $\Res_{L/\Q}(B)$ is isogenous over $\Q$ to  a product
$A_1^{n_1}\times \cdots \times A_r^{n_r}$, where the $A_i$ are
non-isogenous abelian varieties of $\GL_2$-type. But each of these
varieties of $\GL_2$-type $A_i$ is isogenous to some modular
abelian variety $A_{f_i}$, and by \cite[Lemma 5.1]{guitart-quer}
each $n_i$ is equal to $t$.
\end{proof}
The rest of the section is devoted to give an explicit expression
for the modular forms that appear in \eqref{eq: restricion de
escalares B como producto de Afs}, in terms of a certain action of
$\Gal(E/F)$ on a group of characters that we now define. For $s\in
\Gal(E/F)$ let $\chi_s$ be the Galois character such that $\acc s
f =f\otimes\chi_s$, and let $G\subseteq \Hom(G_\Q,\C^\times)$ be
the group generated by all such $\chi_s$. Since $L$ is the fixed
field of $\Qb$ by $\cap_{s} \ker \chi_s$, it is also the fixed
field of $\Qb$ by $\cap_{\chi\in G}\ker \chi$. However, we remark
that a character $\chi\in G$ is not necessarily of the form
$\chi_s$, but a product of elements of the form $\chi_s$ in
general. An element $\chi\in G$ is trivial when restricted to
$\Gal(\Qb/L)$, so it can be identified with a character
$\Gal(L/\Q)\ra \C^\times$.  In fact, $G$ can be identified with
$\Hom(\Gal(L/\Q),\C^\times)$ and therefore we have that $\mid G
\mid =[L:\Q]$ (cf. \cite[pp. 21-22]{Wa}). We define an action  of
$\Gal(E/F)$ on $G$ by
\begin{equation*}
\begin{array}{rcc}
\Gal(E/F)  \times  G & \longrightarrow & G\\
(s ,  \chi)&\longmapsto & s\cdot \chi=\chi_s{^s\chi}.
\end{array}
\end{equation*}
The cocycle identity of the characters $\chi_s$ (\cite[Proposition
3.3 ]{ribet-twists}) implies that it is indeed a group action,
since for $s,t\in \Gal(E/F)$ we have that $s\cdot (t\cdot
\chi)=s\cdot(\chi_t \acc t \chi)=\chi_s\acc s \chi_t
\acc{st}\chi=\chi_{st}\acc{st}\chi=(st)\cdot \chi$. Let $\hat G$
be a system of representatives for the orbits of $G$, and for
$\chi\in G$ let $I_\chi$  be the isotropy subgroup of $G$ at
$\chi$.
\begin{lemma}\label{lema: desigualdad dimension Afxchi}
$\dim (A_{f\otimes\chi})\leq [\Gal(E/F):I_\chi][F:\Q]$
\end{lemma}
\begin{proof}
For $s\in \Gal(E/F)$ the character $\chi_s$ takes values in $E$,
since $\chi_s(p)=\acc s a_p/a_p$ for almost all $p$. Therefore,
any $\chi\in G$ also takes values in $E$. This implies that
$E_{f\otimes\chi}$, the field of Fourier coefficients of
$f\otimes\chi$, is contained in $E$. For any $s\in I_\chi$ we have
that $\acc s(f\otimes\chi)=\acc s f\otimes \acc s \chi=f\otimes
\chi_s\acc s \chi_s=f\otimes (s\cdot \chi )=f\otimes\chi$. By
Galois theory we find  that $I_\chi\subseteq
\Gal(E/E_{f\otimes\chi})$ and so $\mid I_\chi\mid \leq
[E:\Q]/[E_{f\otimes\chi}:\Q]$, which gives that
\begin{eqnarray*}
\dim (A_{f\otimes\chi})=[E_{f\otimes\chi}:\Q]\leq
\frac{[E:\Q]}{\mid I_\chi\mid}=\frac{[E:F]}{\mid
I_\chi\mid}[F:\Q]=[\Gal(E/F):I_\chi][F:\Q].
\end{eqnarray*}
\end{proof}
Now we can give explicitly the modular forms appearing in
proposition \ref{prop: restriccion de escalares de B primera
proposicion}, and we can also give  the dimension of the
corresponding modular abelian varieties.
\begin{proposition}\label{Res}
The inequality in lemma \ref{lema: desigualdad dimension Afxchi}
is in fact an equality and
\begin{equation}\label{eq: formula restriccion de escalares la buena}
\Res_{L/\Q}(B)\sim_\Q \prod_{\chi\in \hat G}\left(A_{f\otimes
\chi}\right)^t.
\end{equation}
\end{proposition}
\begin{proof}
Let $A=\Res_{L/\Q} B$. Since $B$ is $L$-isogenous to all of its
Galois conjugates and $A\simeq_L \prod_{\sigma\in \Gal(L/\Q)}
\acc\sigma B $, we have that $A$ is $L$-isogenous to $B^{[L:\Q]}$.
Therefore, if $A_g$ is a simple factor of $A$ over $\Q$ it
isogenous to a power of $B$ over $L$. Since $A_f$ is also
isogenous to a power of $B$ over $L$, proposition \ref{corolario
ribet}  implies that $g=f\otimes\chi$ for some Galois character
$\chi\in G$. Hence, the modular forms $f_i$ of the decomposition
\eqref{eq: restricion de escalares B como producto de Afs} are of
the form $f\otimes\chi$ for some $\chi$ belonging to $G$. But if
$\chi,\chi'\in G$ are in the same orbit for the action of
$\Gal(E/F)$, the varieties $A_{f\otimes\chi}$ and
$A_{f\otimes\chi'}$ are isogenous over $\Q$. Indeed, in this case
$s\cdot \chi=\chi'$ for some  $s\in\Gal(E/F)$, and then $\acc s
(f\otimes\chi)=f\otimes\chi'$. This, together with the fact that
the modular abelian variety attached to $\acc s(f\otimes \chi)$ is
isogenous over $\Q$ to the one attached to $ f\otimes\chi$ implies
that $A_{f\otimes\chi}$ is isogenous to $ A_{f\otimes\chi'}$ over
$\Q$.

Therefore, proposition \ref{prop: restriccion de escalares de B
primera proposicion} implies that there exists an exhaustive
morphism over $\Q$
$$\lambda\colon\prod_{\chi\in \hat G}\left( A_{f\otimes \chi}\right)^t\lra
A,$$ so we have that
\begin{eqnarray*}
t[F:\Q]|G|&=&|G|\dim B=\dim A\leq \sum_{\chi\in\hat G} \dim
\left(A_{f\otimes\chi}\right)^t\\ &\leq& t[F:\Q]\sum_{\chi\in \hat
G}[\Gal(E/F):I_\chi]=t[F:\Q]|G|.
\end{eqnarray*}
We see that each inequality is in fact an equality, and $\lambda$
is an isogeny since the dimensions of the source and the target
are the same.
\end{proof}

\section{Level-conductors formulas}\label{sec: 3}
As in the previous section  we consider a newform $f\in
S_2(N,\varepsilon)$ and a decomposition $A_f\sim_L B^n$ of $A_f$
into a power of a building block $B$ defined over the field $L$,
and we continue with the same notation as before with respect to
the endomorphism algebra of $B$; namely, $F$ is its center and $t$
its Schur index. In this section we use the decomposition
\eqref{eq: formula restriccion de escalares la buena} to compute
the local exponent of the conductor of $B$. In some particular
cases we prove that the conductor belongs to $\Z$ (i.e. it is a
principal ideal generated by a rational integer), and we are able
to give a global formula for it involving the conductor of $L$ and
the level of $f$. We denote by $\cN_L(B)$ the conductor of $B$
over $L$,  by $\f_L$ the conductor of $L$ and by $N_{L/\Q}$ the
norm in the extension $L/\Q$. If $\chi$ is a character belonging
to $G$, we also denote by $N_\chi$ the level of the newform
$f\otimes\chi$ and by $\f_\chi$ the conductor of $\chi$. For a
prime $q$, $v_q(x)$ denotes the valuation of $x$ at $q$ and
$\chi_q$ the $q$-primary component of $\chi$.

\begin{proposition}\label{cond_local}
For each rational prime $q$ we have that
\begin{equation}\label{eq: key equation para v_q}
v_q(N_{L/\Q}(\cN_L(B)))+2\,(\dim B)\,\sum_{\chi\in G}
v_q(\f_\chi)=(\dim B) \sum_{\chi\in G}v_q(N_\chi).
\end{equation}
\end{proposition}
\begin{proof}
By applying the formula of \cite[Proposition 1]{milne} for the
conductor of the restriction of scalars to \eqref{eq: formula
restriccion de escalares la buena} we obtain that
\begin{equation*}
N_{L/\Q}\left(\cN_L(B)\right)(d_{L/\Q})^{2\dim B}=\prod_{\chi\in
\hat G}\cN_\Q\left(A_{f\otimes\chi}\right)^t,
\end{equation*}
where $d_{L/\Q}$ is the discriminant of $L/\Q$. By a theorem of
Carayol (\cite{carayol}) the conductor of a modular abelian
variety $A_g$ is $N_g^{\dim A_g}$, where $N_g$ is the level of the
newform $g$. Using this property and the conductor-discriminant
formula (cf. \cite[p. 28]{Wa}) we find that
\begin{equation*}
N_{L/\Q}\left(\cN_L(B)\right)\prod_{\chi\in G}(\f_\chi)^{2\dim B
}=\prod_{\chi\in \hat G}N_{\chi}^{t[\Gal(E/F):I_\chi][F:\Q]}.
\end{equation*}
But $t[F:\Q]=\dim B$, and the orbit of $\chi$ contains
$[\Gal(E/F):I_\chi]$ elements, each one giving a modular abelian
variety of the same dimension. Thus we have that
\begin{equation*}
N_{L/\Q}\left(\cN_L(B)\right)\prod_{\chi\in G}(\f_\chi)^{2\dim B
}=\prod_{\chi\in  G}N_\chi^{\dim B}
\end{equation*}
from which \eqref{eq: key equation para v_q} follows by taking
valuations at $q$.
\end{proof}
Each prime $\mathfrak q$ dividing $q$ appears in $\cN_L(B)$ with
the same exponent (see the proof of lemma \ref{lemma: integrality
of the conductor}). This observation together with \eqref{eq: key
equation para v_q}  gives a way of computing  the local exponents
of $\cN_L(B)$ in terms of the levels $N_\chi$. In almost all cases
\cite[Theorem 3.1]{Atkin-Li} can be used to compute the levels of
the twisted newforms. Under some hypothesis one can also perform
directly the computation, as in the following:
\begin{lemma}\label{lemma: formula ord eps leq 2 q neq 2}
If $\ord(\varepsilon)\leq 2$ then for all primes $q\neq 2$ we have
that
\begin{equation}\label{eq: formula ord eps leq 2 q neq 2}
v_q(N_{L/\Q}(\cN_L(B)))+ [L:\Q]\,(\dim B)\,v_q(\f_L)=[L:\Q]\,(\dim
B)\,v_q(N).
\end{equation}
\end{lemma}
\begin{proof}
First of all suppose that $v_q(\f_L)=0$. Then for each $\chi\in G
$ we have that $v_q(\f_\chi)=0$ and \eqref{eq: formula ord eps leq
2 q neq 2} follows from  \eqref{eq: key equation para v_q} since
$v_q(N_\chi)=v_q(N)$ for all $\chi\in G$ and $|G|=[L:\Q]$.

Suppose now that $v_q(\f_L)\neq 0$. This means that there exists
an element $s\in\Gal(E/F)$ such that $v_q(\f_{\chi_s})\neq 0$. But
$\chi_{s}$ is a quadratic character since the relation $\acc s
f=f\otimes\chi_s$ implies that $\chi_s^2=\acc s
\varepsilon/\varepsilon=1$. So $\chi_{s,q}$, the $q$-primary part
of $\chi_s$, is the unique character of order 2 and conductor a
power of $q$, that we denote by $\xi_q$. Since $\xi_q$ has
conductor $q$ we see that $v_q(\f_L)=1$. For $i=0,1$ define
$G_q^i=\{\chi\in G \,| \chi_q=\xi_q^i \}$. We have that
$G=G_q^0\sqcup G_q^1$ and that the map $\chi\mapsto \chi\chi_s$ is
a bijection between $G_q^0$ and $G_q^1$. Hence
$|G_q^0|=|G_q^1|=|G|/2$. For $\chi\in G_q^0$ we have that
$v_q(\f_\chi)=0$ and $v_q(N_\chi)=v_q(N)$. For $\chi\in G_q^1$ we
have that $v_q(\f_\chi)=1=v_q(\f_L)$ and
$v_q(N_\chi)=v_q(N_{\chi_s})=v_q(N)$, because the level of
$f\otimes\chi_s$ is the level of $\acc s f$ which is $N$. Plugging
all this into \eqref{eq: key equation para v_q} one obtains
\eqref{eq: formula ord eps leq 2 q neq 2}.
\end{proof}
\begin{lemma}\label{lemma: integrality of the conductor}
$\cN_L(B)$ belongs to $\Z$ if and only if $[L:\Q]$ divides
$v_q(N_{L/\Q}(\cN_L(B)))$ for all rational primes $q$.
\end{lemma}
\begin{proof}
Suppose that $q$ decomposes in $L$ as $\mathfrak q_1^e\mathfrak
q_2^e\cdots \mathfrak q_g^e$. For each $\sigma\in \Gal(L/\Q)$ we
have that $\acc\sigma B$ is $L$-isogenous to $B$, so
$\cN_L(B)=\cN_L(\acc\sigma B)=\acc\sigma\cN_L(B)$. This means that
if  $\mathfrak q_1^n$ exactly divides $\cN_L(B)$ then the rest of
$\mathfrak q_i^n$ also exactly divide $\cN_L(B)$ and then
$v_q(N_{L/\Q}(\cN_L(B)))=nfg$, where $f$ denotes the residual
degree of $\mathfrak q_i$. Now $\cN_L(B)$ belongs to $\Z$ if and
only if for all primes $q$ the exponent $e$ divides $  n$, and
because of the relation $efg=[L:\Q]$ this is equivalent to the
fact that $v_q(N_{L/\Q}(\cN_L(B)))$ is divisible by $[L:\Q]$.
\end{proof}

\begin{proposition}\label{prop: formula conductor nivel ord leq 2}
If $N$ is odd and $\ord(\varepsilon)\leq 2$ then $\cN_L(B)$
belongs to $\Z$ and
\begin{equation}
\cN_L(B)\ \f_L^{\,\dim B}=N^{\,\dim B}.
\end{equation}
\end{proposition}
\begin{proof}
By lemma \ref{lemma: formula ord eps leq 2 q neq 2} for each
prime $q$ we have that $v_q(N_{L/\Q}(\cN_L(B)))$ is multiple of
$[L:\Q]$, and by lemma \ref{lemma: integrality of the conductor}
this implies that $\cN_L(B)$ belongs to $\Z$. In consequence,
$N_{L/\Q}(\cN_L(B))=\cN_L(B)^{[L:\Q]}$, and using this in
\eqref{eq: formula ord eps leq 2 q neq 2} we have that
$v_q(\cN_L(B))+(\dim B)\,v_q(\f_L)=(\dim B)\, v_q(N).$ Since this
holds for all $q$, the proposition follows.
\end{proof}
\begin{remark}
If $\dim A_f=2$ and $\ord(\varepsilon)\leq 2$, then either $A_f$
is absolutely simple or it is isogenous over a quadratic number
field $L$ to the square of a $\Q$-curve $B/L$. In the second case it is always true that
$\cN_L(B)$ belongs to $\Z$ and $\cN_L(B)\f_L=N$. This follows by
applying Milne's formula  to the restriction of scalars of $B$,
for which  we have that $\Res_{L/\Q}(B)\sim_\Q A_f$.
\end{remark}

Proposition \ref{prop: formula conductor nivel ord leq 2} might be
seen as a  generalization of Carayol's formula
$\cN_\Q(A_f)=N^{\dim A_f}$  for modular abelian varieties. As we
will see this formula does not generalize to arbitrary newforms;
in other words, our hypothesis on the parity of $N$ and on the
order of the character are necessary. However, for modular forms
on $\Gamma_0(N)$ and with arbitrary $N$ it is still true except
for a factor $2$, that appears or not depending on the field $L$.

\begin{proposition}\label{prop: formula conductor nivel Gamma0}
Suppose that $\varepsilon=1$. Then $\cN_L(B)$ is an integer and
\begin{enumerate}
\item $2\ \cN_L(B)\ \f_L^{\,\dim B}=N^{\dim B}$\  if $v_2(\f_L)=3$ and $v_2(\f_K)=2$ for some $K\subseteq L$.
\item $\cN_L(B)\ \f_L^{\,\dim B}=N^{\dim B}$ \ otherwise.

\end{enumerate}
In particular, if $v_2(N)\leq 4$ the second formula holds.
\end{proposition}
\begin{proof}
Since $\varepsilon$ is trivial the character $\chi_s$ is quadratic
for all $s\in \Gal(E/F)$, so that any $\chi\in G$ is quadratic.
Define the set $P_2=\{\chi_2\,|\, \chi\in G\}$, which has cardinal
$\leq 4$ because the set of quadratic characters of conductor a
power of $2$ is isomorphic to $\Z/2\Z\times \Z/2\Z$. Observe that
the condition $v_2(\f_L)=3$ and $v_2(\f_K)=2$ for some $K\subseteq
L$ is equivalent to $|P_2|=4$. We begin by proving the second
formula in the statement, which corresponds to the case $|P_2|\leq
2$.

If $2\nmid \f_L$ then $|P_2|=1$ and
\begin{equation}\label{eq: formula v_2 q nmid}
v_2(N_{L/\Q}(\cN_L(B)))+ [L:\Q]\,(\dim B)\,
v_2(\f_L)=[L:\Q]\,(\dim B)\,v_2(N).
\end{equation}
If $2\mid \f_L$,  then for each $s\in \Gal(E/F)$ the character
$\chi_{s,2}$ is either trivial or quadratic, so that the set $P_2$
can have cardinal $2$ or $4$. Suppose first that $|P_2|=2$, and
fix an $s$ such that $\chi_{s,2}\in P_2$. Then for $i=0,1$ define
$G_2^i=\{\chi\in G\, | \, \chi_2=\chi_{s,2}^i\}$. Observe that if
$\chi\in G_2^0$ then $v_2(\f_\chi)=0$ and if $\chi\in G_2^1$ then
$v_2(\f_\chi)=v_2(\f_L)$. Moreover, for all $\chi\in G$ we have
that $v_2(N_\chi)=v_2(N)$, since
$v_2(N_\chi)=v_2(N_{\chi_s})=v_2(N)$. Now with the same reasoning
as in lemma \ref{lemma: formula ord eps leq 2 q neq 2} we find
that \eqref{eq: formula v_2 q nmid} also holds in this case. This,
together with lemma \ref{lemma: formula ord eps leq 2 q neq 2}
implies that the formula
\begin{equation}\label{eq: formula for q neq 2 inside a proposition}v_q(N_{L/\Q}(\cN_L(B)))+
[L:\Q]\,(\dim B)\, v_q(\f_L)=[L:\Q]\, (\dim B)\,
v_q(N)\end{equation} is true for all $q$. Arguing as in the proof
of proposition \ref{prop: formula conductor nivel ord leq 2} this
implies the second formula of the statement.

Suppose now that $|P_2|=4$. If we denote by $\xi$ and $\psi$ the
quadratic characters of conductor $8$, then $\xi\psi$ is the
quadratic character of conductor $4$ and
$P_2=\{1,\xi,\psi,\xi\psi\}$. Define $G_\xi=\{\chi\in G \, |\,
\chi_2=\xi\}$, and similarly for the other characters define
$G_{\psi}$, $G_{\xi\psi}$ and $G_1$. Each one of these sets has
cardinal $|G|/4$. If $\chi$ belongs to $G_\xi$ or $G_{\psi}$, then
$v_2(\f_\chi)=3$, while if $\chi$ belongs to $G_{\chi\psi}$ then
$v_2(\f_\chi)=2$. Therefore the relation \eqref{eq: key equation
para v_q} gives
$$v_2(N_{L/\Q}(\cN_L(B)))+2\,(\dim B)\,\left( 2\frac{|G|}{4}+3\frac{|G|}{4}+3\frac{|G|}{4} \right)=|G|\,(\dim B)\,v_2(N),$$
and since now $v_2(\f_L)=3$ we arrive at
$$v_2(N_{L/\Q}(\cN_L(B)))+[L:\Q]\,(\dim B)\,(v_2(\f_L)+1)=[L:\Q]\,(\dim
B)\,v_2(N).$$We see that in this case $v_2(N_{L/\Q}(\cN_L(B)))$ is
also multiple of $[L:\Q]$. As before, for $q\neq 2$ formula
\eqref{eq: formula for q neq 2 inside a proposition} also holds in
this case, so we conclude that now $2\,\cN_L(B)\,\f_L=N^{\dim B}$.

To prove the last statement, let $s\in \Gal(E/F)$ and let $\chi_s$
be the corresponding quadratic character. Then
$v_2(N_{\chi_s})=v_2(N)$, and if $\varepsilon=1$ and $v_2(N)\leq 4
$ by \cite[Theorem 3.1]{Atkin-Li} this is not possible  if the
conductor of $\chi_{s,2}$ is $8$. Therefore  $\chi_{s,2}$ is
either  trivial or the quadratic character of conductor $4$, and
we see that $|P_2|\leq 2$.
\end{proof}

We remark that the first case in (\ref{prop: formula conductor
nivel Gamma0}) does occur. For instance, let $f$ be the unique (up
to conjugation) normalized newform for $\Gamma_0(512)$ such that
$A_f$ has dimension $4$. Using  {\tt Magma} \cite{magma} one can
compute the characters associated to the inner twists of $f$; it
turns out that  some of them have conductor divisible by $8$ and
some of them have conductor exactly divisible by $4$ and therefore
$|P_2|=4$.

\begin{proposition}\label{prop: squarefree}
If $N$ is squarefree then for all primes $q$ dividing $N$ we have
that
\begin{enumerate}
\item $v_q(N_{L/\Q}(\cN_L(B)))=(\dim B)[L:\Q]$ if
$q\nmid\f_\varepsilon$.
\item $v_q(N_{L/\Q}(\cN_L(B)))=0$ if $q\mid
\f_\varepsilon$.
\end{enumerate}
In particular, $\cN_L(B)$ belongs to $\Z$ and
\begin{equation}\label{eq: conductor formula N squarefree}\cN_L(B)\, \f_L^{\,\dim B}=N^{\dim
B}.\end{equation}
\end{proposition}
\begin{proof}
If $q\nmid\f_\varepsilon$ then $v_q(\f_\chi)=0$ for all $\chi\in
G$ and then the formula follows easily from \eqref{eq: key
equation para v_q}.

Suppose that $q\mid \f_\varepsilon$. Let $s\in\Gal(E/F)$ and let
$\chi_s$ be the character such that $\acc s f=f\otimes\chi_s$.
Observe that  $v_q(N)=v_q(N_{\chi_s})$, and by \cite[Theorem
3.1]{Atkin-Li}  under our hypothesis this is possible if and only
if $\chi_{s,q}=1$ or $\chi_{s,q}=\varepsilon^{-1}_q$. This means
that for each $\chi\in G$, the character $\chi_q$ is of the form
$\varepsilon_q^i$ for some $i$. In particular,
$v_q(\f_L)=v_q(\f_\varepsilon)$. Let $n=\ord(\varepsilon_q)$, and
for $i=0,\dots,n-1$ define $G_q^i=\{\chi\in G \, |  \,
\chi_q=\varepsilon_q^i\}$. The map $\chi\mapsto \chi\varepsilon$
is a bijection between $G_q^i$ and $G_q^{i+1}$, and since
$G=\sqcup_{i=0}^{n-1} G_q^i$ we see that $|G_q^i|=|G|/n$.

If $\chi\in G_q^i$ then $v_q(\f_\chi)=1$ for $i=1,\dots, n-1$,
while $v_q(\f_\chi)=0$ for $i=0$. If $\chi\in G_q^i$ for $i=0,n-1$
then $v_q(N_\chi)=v_q(N)$; if $\chi\in G_q^0$ this is clear, and
if $\chi\in G_q^{n-1}$ this is because
$v_q(N_\chi)=v_q(N_{\varepsilon_q^{-1}})=v_q(N_{\varepsilon^{-1}})=v_q(N)$
since $\overline f=f\otimes\varepsilon^{-1}$. On the other hand,
for the rest of the values $i=2,\dots,n-2$ then $v_q(N_\chi)=2$ if
$\chi\in G_q^i$; this follows from \cite[Theorem 3.1]{Atkin-Li}.
Gathering all this information we can rewrite \ref{eq: key
equation para v_q} in this case as
\begin{equation*}
v_q(N_{L/\Q}(\cN_L(E)))+2\,(\dim B)\,\sum_{i=0}^{n-1}\sum_{\chi\in
G_q^i} v_q(\f_\chi)=(\dim B)\, \sum_{i=0}^{n-1}\sum_{\chi\in
G_q^i}v_q(N_\chi),
\end{equation*}
and this gives
$$v_q(\M)+2\,(\dim B)\,\frac{|G|}{n}(n-1)=(\dim
B)\,\left(|G|+\frac{|G|}{n}(n-2)\right),$$ which is directly the
formula for the second case.

Finally,  the two formulas in the statement can be written as the
following expression, which is valid for all $q$:
$$v_q(N_{L/\Q}(\cN_L(B)))+[L:\Q]\,(\dim B)\,v_q(\f_L)=[L:\Q]\,(\dim B)\,
v_q(N).$$ This implies that $\cN_L(B)$ belongs to $\Z$ and also
the formula \eqref{eq: conductor formula N squarefree}.
\end{proof}
\begin{remark}
Observe that for squarefree $N$ proposition \ref{prop: squarefree}
completely characterizes the places of good reduction of $B$:  a
prime $\mathfrak q\mid q$ is a prime of good reduction of $B$ if
and only if $q\nmid N$  or $q\mid \f_\varepsilon$.
\end{remark}

\section{Examples}\label{sec: 4}
In this section we show explicit examples of building blocks of
dimension one and two where we compute their conductors. As in the
rest of the paper, all the  newforms we consider are without
complex multiplication.
\subsection{$\Q$-curves}
All the examples in this paragraph come from modular abelian
varieties $A_f$ where the corresponding building block $B$ has
dimension one. An algorithm to compute equations of building
blocks of dimension one is provided by Gonz\'alez and Lario in
\cite{pep-lario}. The equations for the first three examples were
computed using that algorithm, and  the equations for examples
\ref{ex64} and \ref{ex81} have been provided by Jordi Quer.
The conductor of these elliptic curves have been computed using
\verb+Magma+.
\begin{example}{\rm
Let $f$ be the unique (up to conjugation) normalized newform of
weight two, level $42$ and Nebentypus of order $2$ and conductor
$21$. We have  $\dim A_f=4$ and $L=\Q(\sqrt{-3},\sqrt{-7})$. In
\cite{pep-lario} it is proved that an equation for $B$ is given by
$$
B:y^2=x^3+\frac{81}{4}(69+43\sqrt{-3}+29\sqrt{-7}+17\sqrt{21})x+162(207-84\sqrt{-3}-54\sqrt{-7}+46\sqrt{21}).
$$
We have $\cN_L(B)=2$ and $\mathfrak{f}_L=21$. Therefore we have
that the Proposition \ref{prop: formula conductor nivel ord leq 2}
holds although $N$ is even in this case. }\end{example}

\begin{example}\label{ex64}{\rm
Let $f$ be the unique (up to conjugation) normalized newform of
weight two, level $64$ and Nebentypus of order $4$ and conductor
$16$. In this case, $A_f$ is an abelian surface and
$L=\Q(\alpha)$, where $\alpha^4-4\alpha^2+2=0$. An equation for
$B$ is
$$
B:y^2=x^3-432(5-8\alpha+14\alpha^2-6\alpha^3)x-864(-124+74\alpha+194\alpha^2-107\alpha^3).
$$
We have $\cN_L(B)=2$ and $\mathfrak{f}_L=16$. Therefore we have
that $2\ \cN_L(B)\ \f_L^{\,\dim B}=N^{\dim B}$. }\end{example}
\begin{example}\label{ex81}{\rm
Let $f$ be the unique (up to conjugation) normalized newform
(without complex multiplication) of weight two, level $81$ and
Nebentypus of order $3$ and conductor $9$. In this case, $\dim
A_f=4$ and $L=\Q(\sqrt{-3},\alpha)$, where $\alpha^3-3\alpha+1=0$.
An equation for the building block is:
$$
B:y^2=x^3-\frac{81}{2}(-54+14\sqrt{-3}+2
(12+5\sqrt{-3})\alpha+(27-7\sqrt{-3})
\alpha^2)x+729(37+19\sqrt{-3}).
$$
We have $\cN_L(B)=3$ and $\mathfrak{f}_L=9$. Therefore we have
that $3\ \cN_L(B)\ \f_L^{\,\dim B}=N^{\dim B}$. }\end{example}

In the above examples the conductors $\cN_L(B)$ turned out to be
rational integers. The following example shows that this is not
always the case.
\begin{example}\label{ex98}{\rm
There are two normalized newforms of weight two, level $98$ and
Nebentypus $\varepsilon$ of order $3$ and conductor $7$ such that
the the associated abelian variety is a surface. In both cases
$L=\Q(\alpha)$, where $\alpha^3 + \alpha^2 - 2\alpha - 1=0$. To
obtain an equation for the building block we are going to proceed
in a different way than above. Let $f$ be one of these two
newforms. Then, $A_f\sim_L B^2$ and by Proposition \ref{Res} we
have that $\Res_{L/\Q}(B)\sim_\Q A_f\times
A_{f\otimes\varepsilon}$. Therefore $\dim
A_{f\otimes\varepsilon}=1$. In particular, $B\sim_L
A_{f\otimes\varepsilon}$. Then, instead of computing an equation
of $B/L$ using the Gonz\'alez-Lario  algorithm \cite{pep-lario} we
are going to compute an equation of $A_{f\otimes\varepsilon}$ over
$\Q$. In this case the results of Atkin and Li \cite{Atkin-Li} do
not provide the exact level of $f\otimes\varepsilon$, althought
they assert that the level is a divisor of $98$ of the form
$2\cdot 7^n$. One of the twisted newforms corresponds to the
unique (up to $\Q$-isogeny) elliptic curve defined over $\Q$ of
level $14$ and the other one to the unique (up to $\Q$-isogeny)
elliptic curve defined over $\Q$ of level $98$ (they are labelled
as \verb+14A+ and \verb+98A+  respectively in Cremona's tables
\cite{cremona} or in Antwerp tables \cite{antwerp}). The equations
for these building blocks are:
$$
\begin{array}{crcl}
B_1\,:\, & y^2 + x y + y & = & x^3 + 4 x - 6,\\
B_2\,:\, & y^2 + x y & = & x^3 + x^2 - 25 x - 111.
\end{array}
$$
For $i=1,2$, we have $\cN_L(B_i)= 2  (\alpha^2 - \alpha -
2)^i\cdot\mathcal{O}_L$, which is not an ideal generated by a
rational integer. 
}
\end{example}

\subsection{Genus $2$ curves}
In the opposite to the elliptic curve case, there is no
implementation of an algorithm to compute the conductor of a genus
$2$ curve over a number field. For that purpose, Proposition
\ref{cond_local} allows us to compute the conductor of a genus $2$
curve that corresponds to the building block of a modular abelian
variety.
\begin{example}\label{g2_xevi}{\rm
Let $C$ be the genus $2$ curve defined over $\Q(\sqrt{-6})$
defined by
$$
\begin{array}{ll}
C\,:\,y^2= & \!\!\!\!\!(27\sqrt{-6} - 324)x^6 - 15876x^5 + (-7938\sqrt{-6} - 222264)x^4 - 345744x^3 \\
& \!\!+ (-259308\sqrt{-6} + 7260624)x^2 - 16941456x +
941192\sqrt{-6}  + 11294304.
\end{array}
$$
The jacobian $B={\rm Jac}(C)$ is a building block with
quaternionic multiplication, and in  \cite{guitart-quer} it is
proved that $\Res_{L/\Q}(B)$ is isogenous to a product of modular
abelian varieties over $\Q$, where $L=\Q(\sqrt{2},\sqrt{-3})$. In
fact, there are numerical evidences  suggesting  that
$\Res_{L/\Q}(B)\sim_\Q A_f^2$, where  $f$ is a newform of level
$N=2^8 3^5$ and Nebentypus $\varepsilon$ of order $2$ and
conductor $8$ with $\dim A_f=4$. Assuming this we are going to
prove that $\cN_L(B)=2^{10} 3^8$, and therefore the formula
$\cN_L(B)\ \f_L^{\,\dim B}=N^{\dim B}$ would hold again. The
subfields of $L$ are the quadratic fields $\Q(\sqrt{2})$,
$\Q(\sqrt{-3})$ and $\Q(\sqrt{-6})$, and they correspond to the
non-trivial homomorphisms from  $G$ to $\C^\times$. Therefore the
conductors of these quadratic fields correspond to the conductors
of the non-trivial elements at $G$. On the other hand, by the
decomposition of $\Res_{L/\Q}(B)$ given above and by \cite[Theorem
3.1]{Atkin-Li} we have that  $N_\chi=N$ for all $\chi\in G$. Now,
 Proposition \ref{cond_local} at the primes $q=2$ and $q=3$ give
us the stated conductor of $B$ over $L$. }\end{example}

\subsection*{Acknowledgements}
We thank A. Brumer for useful discussion on conductors of abelian
varieties and J. Quer for providing us equations for some of the
examples of Section \ref{sec: 4}.


\end{document}